\documentclass[runningheads,a4paper]{llncs}
\usepackage{amsmath,amssymb}
\usepackage{graphicx}
\usepackage{url}
\makeatletter
\def\blfootnote{\gdef\@thefnmark{}\@footnotetext}
\makeatother

\urldef{\mailsb}\path| oguz.yayla@hacettepe.edu.tr|
\urldef{\mailsa}\path| busra.ozden@hacettepe.edu.tr|

\newcommand{\keywords}[1]{\par\addvspace\baselineskip
\noindent\keywordname\enspace\ignorespaces#1}

\usepackage[left=3cm,top=3cm,right=3cm,bottom=2cm]{geometry}
\usepackage{appendix}
\usepackage[pagewise]{lineno}

\usepackage{tikz}
 
\usetikzlibrary{arrows,automata,backgrounds,fit,shapes,decorations.pathreplacing}

\RequirePackage[T1]{fontenc}
\RequirePackage[utf8]{inputenc}

\renewenvironment{proof}{\noindent\textit{Proof.}}{\hfill{$\Box$}}

\newcommand{\boxtensor}{{\Box\kern-9.03pt\raise1.42pt\hbox{$\times$}}}

\newcommand{\C}{{\mathbb C}}

\newcommand{\N}{{\mathbb N}}

\newcommand{\Z}{{\mathbb Z}}

\newcommand{\be}{\begin{eqnarray}}
\newcommand{\ee}{\end{eqnarray}}
\newcommand{\nn}{{\nonumber}}
\newcommand{\dd}{\displaystyle}

\begin{document}

\mainmatter  

\title{Partial direct product difference sets and sequences with ideal autocorrelation}

\titlerunning{PDPDS and sequences with ideal autocorrelation}

\author{Büşra Özden \and O\u guz Yayla}
\authorrunning{Özden, Yayla}

\institute{Hacettepe University, Department of Mathematics\\Beytepe, 06800, Ankara, Turkey\\ \mailsa\\
\mailsb}

\toctitle{Lecture Notes in Computer Science}
\tocauthor{Authors' Instructions}
\maketitle

\begin{abstract}
In this paper, we study the sequences  with (non-consecutive) two zero-symbols and ideal autocorrelation, which are also known as almost $m$-ary nearly perfect sequences. We show that these sequences are equivalent to $\ell$-partial direct product difference sets (PDPDS), then we extend known results on the sequences with two consecutive zero-symbols to non-consecutive case.  
Next, we study the notion of multipliers and orbit combination for  $\ell$-PDPDS. 
Finally, we present a  construction method for a family of almost quaternary sequences with ideal autocorrelation by using cyclotomic classes. 
\blfootnote{\textup{2020} \textit{Mathematics Subject Classification}: \textup{05B10, 94A55.}}
\end{abstract}

\section{Introduction}
Let $\zeta_m\in \C$ be a primitive $m$-th root of unity for some integer $m$.
	A sequence $\underline{a} = (a_0,a_1,\ldots,a_{n-1},\ldots)$ of period $n$ with $a_i = \zeta_m ^{b_i}$ for some integer $b_i$, $i =0, 1, \ldots,n-1$ is called an \textit{$m$-ary sequence}. If $a_{i_j} = 0$ for all $j=1,2,\ldots , s$ where $\{ i_1,i_2, \ldots , i_s \} \subset \{0,1, \ldots , n-1\}$ and $a_i = \zeta_p^{b_i}$ for some integer $b_i$, $i \in  \{0,1, \ldots , n-1\} \backslash \{ i_1,i_2, \ldots , i_s \}$, then we call $\underline{a}$ an \textit{almost $m$-ary sequence with $s$ zero-symbols}. 
	For instance, $\underline{a}=(\zeta_6^4,1,\zeta_6^2,\zeta_6^5,1,\zeta_6,
	\ldots)$ is a $6$-ary sequence of period $6$ and  $\underline{a}=(0,\zeta_4^3,1,\zeta_4^2,0,0,\zeta_4,1,\zeta_4^2,\ldots)$ is an almost $7$-ary sequence with 3 zero-symbols of period $9$.  It is widely used that a sequence with one zero-symbol is called an almost $m$-ary sequence. But in this paper we use this notation for a $m$-ary sequence with $s$ zero-symbols, for $s \geq 0$. 
	
	For a sequence $\underline{a}$ of period $n$, its \textit{autocorrelation 
		function} $C_{\underline{a}}(t)$ is defined as
	\be \label{df:autocor}
	C_{\underline{a}}(t) = \sum_{i=0}^{n-1}{a_i\overline{a_{i+t}}},
	\ee
	for $0\leq t \leq n-1$ where $\overline{a}$ is the complex conjugate of $a$.  The values $C_{\underline{a}}(t)$ at $1 \leq t \leq n-1$ are called \textit{the out-of-phase autocorrelation coefficients} of $\underline{a}$. Note that the autocorrelation function of $\underline{a}$ is periodic with $n$.
	
	We call an almost $m$-ary sequence $\underline{a}$ of period $n$ a \textit{nearly perfect
		sequence} (NPS) of type $(\gamma_1, \gamma_2)$ if all out-of-phase autocorrelation coefficients of $\underline{a}$ are either $\gamma_1$ or $\gamma_2$. We write \textit{NPS of type $\gamma$} to denote an NPS of type $(\gamma,\gamma)$. Moreover, a sequence is called \textit{perfect sequence} (PS) if it is an NPS of type $(0,0)$. We also note that there is another notion of \textit{almost perfect sequences} which is an $m$-ary sequence $\underline{a}$ of period $n$ having $C_{\underline{a}}(t)=0$ for all $1\leq t \leq n-1$ -with exactly one exception \cite{jungnickel1999perfect}.

Sequences are used in satellite telecommunication, cryptographic function design, wireless networks, signal processing, radar systems, and modern cell phones (see \cite{ChihLin1995MulticodeCW,golomb2005signal,hollon2018new,kumar,Schmidt2009QuaternaryCC,ozdenyayla2019}). For instance, they are used in Code Division Multiple Access (CDMA), where the sequences with ideal correlation are important because a signal should not be affected by other signals in order to provide high-quality communication. Hence, designing new  $m$-ary sequences with ideal autocorrelation is important and many authors have been studied in this area so far. 

It is known that the  difference sets and the sequences are related topics in combinatorics, and their equivalence was widely studied. For binary case, Jungnickel and Pott~\cite{jungnickel1999perfect} obtained a relation between a binary NPS of type $(n-4(k-\lambda))$ and a cyclic $(n,k,\lambda)$ difference set (DS). For nonbinary case, Ma and Ng
~\cite{ma2009non} obtained a relation between an $m$-ary NPS of type $|\gamma|\leq 1$ and a direct product difference set (DPDS) for a prime number $m$. Later, Chee et al.~\cite{chee2010almost} extended the methods due to Ma and Ng~\cite{ma2009non} to almost $m$-ary NPS of types $\gamma=0$ and $\gamma=-1$ with one zero-symbol. The second author~\cite{yayla2016nearly} proved an equality between an $m$-ary NPS of type $\gamma$  and a DPDS for an arbitrary integer $ \gamma$. In \cite{ozdenyayla2020}, the authors studied the $m$-ary NPS of type $(\gamma_1,\gamma_2)$ with two consecutive zero-symbols, in which they proved that there is a relationship between a partial direct product difference sets (PDPDS) and an $m$-ary NPS of type $(\gamma_1,\gamma_2)$ with two consecutive zero-symbols. Moreover, they showed some non-existence results for these kind of sequences. It is also conjectured that an almost $m$-ary NPS with two consecutive zero-symbols does not exit for $m \ge 5$ (see Conjecture \ref{conj:1}). This motivated us to study the non-consecutive zero-symbols case and the only one-zero-symbol case.

In this paper, we firstly study almost $m$-ary NPS of type $(\gamma_1,\gamma_2)$ with two zero-symbols and extend the definition of PDPDS to $\ell$-PDPDS (see Definition \ref{df:l-pDPDS}). Then we present a method to construct the difference sets $\ell$-$(n,n,2n-2,n-2,n,n-2,2,1)$-PDPDS and  $(n,n,2n-2,n-2,n-2,2)$-DPDS (see Proposition \ref{prop:lpdpds:cons}). Similar to \cite{ozdenyayla2020}, we show that an almost $m$-ary NPS of type $(\gamma_1,\gamma_2)$ with non-consecutive two zero-symbols is equivalent to an $\ell$-PDPDS (see Proposition \ref{th:l_PDPDS}) and their non-existence for $\gamma_2 \le 3$ (see Proposition \ref{prop:non}).
Next, we study the notion of multipliers and the orbit combination for $\ell$-PDPDS, and we prove a necessary condition on their multiplier set (see Proposition \ref{prop:multiplier}). Then, we could obtain a NPS with non-consecutive zero-symbols by a union of some orbits under its multipliers in contrast to consecutive case (see Proposition \ref{prop:orbit} and Example \ref{ex:orbit}). 
Besides, we show that an almost $m$-ary NPS of type $(\gamma_1, \gamma_2)$ and period $n+2$ with two consecutive zero-symbols is symmetric for an odd prime $m$ (see Corollary \ref{cor:sym}). 
We also show a method to get 2m $m$-ary NPS  from one $m$-ary sequence with the same autocorrelation values 
(see Theorem \ref{thm:p_ary-const.}).

In the second part of this paper, we consider the cyclotomic classes  to construct an NPS with only one zero-symbol. 
Shi et. al.~in \cite[Theorem 3]{shi2019family} showed that almost $m$-ary $\sigma$-sequence is an almost $m$-ary NPS of type $-1$ with one-zero-symbol (see also Theorem \ref{thm:m_ary-const11}). 
In this study, we particularly consider the case $q=2p=mf+l$ for a prime $p$, $m,f\in\Z^+$ and $l=p+1$. In this case, we show that the almost $m$-ary $\sigma$-sequence is an almost $m$-ary NPS of type $(0,-1)$ with $p+1$ zero-symbols (see Proposition \ref{thm:m_ary-const1}). We prove that if $f$ is even, then the almost $m$-ary $\sigma$-sequence is symmetric except for $a_0$ (see Proposition \ref{prop:sym}). 
Furthermore, we show that quaternary $\sigma$-sequences for identity permutation $\sigma$ are almost quaternary NPS of type $\frac{q-3}{2}$ with one zero-symbol (see Theorem \ref{thm:m_ary-const2}).

This paper is organized as follows. In Section \ref{sec:pDPDS} we define $\ell$-PDPDS and study its equivalence to an NPS with (non-consecutive) two zero-symbols. Then, we present two construction methods for a family of almost $m$-ary NPS in Section \ref{sec:cons}.

\section{$\ell$-Partial Direct Product Difference Sets}
\label{sec:pDPDS}

It is widely studied in the literature that  any kind of perfect sequence can be identified by a suitably defined difference set. Hence, one can devise the  properties of the perfect sequence via the corresponding difference set, and vice-versa. In this study, we consider almost $m$-ary nearly perfect sequences (NPS) with two zero-symbols non-necessarily consecutive, and their corresponding difference set definition, which we give below. 

\begin{definition}[$\ell$-PDPDS]\label{df:l-pDPDS}
	Let $G = H \times P$, where  $H= \langle h \rangle$ and $P= \langle g \rangle $ of orders $n$ and $m$, respectively. Let  $\ell\in H\backslash\{e_H\} $ be an integer. A subset $R$ of $G$, $|R| =k$, is called an $\ell$-$(n,m,k,\lambda_1,\lambda_2,\lambda_3,\mu_1,\mu_2)$-partial direct product difference set (PDPDS) in $G$ relative to $H$ and $P$ if differences $r_1r_2^{-1}$, $r_1, r_2 \in R$ with $r_1 \neq r_2$ represent
	\begin{itemize}
		\item all elements of $H\backslash\{h^\ell,h^{n-\ell},e_H\}$ exactly $\lambda_1$ times,
		\item all non identity element of $P$ exactly $\lambda_2$ times,
		\item all elements of $\{h^\ell,h^{n-\ell}\}$ exactly $\lambda_3$ times,
		\item all elements of $(H\backslash\{h^\ell,h^{n-\ell},e_H\}) \times (P\backslash\{e_P\})$ exactly $\mu_1$ times,
		\item all elements of $\{h^\ell ,h^{n-\ell}\} \times (P\backslash\{e_P\})$ exactly $\mu_2$ times.
	\end{itemize}
\end{definition}

In the group-ring algebra notation, if $R$ is an $\ell$-$(n,m,k,\lambda_1,\lambda_2,\lambda_3,\mu_1,\mu_2)$-PDPDS in $G$ relative to $H$ and $P$ then
\be\label{eqn:l-PDPDS}
\begin{aligned}
	RR^{(-1)} &= (k - \lambda_1 - \lambda_2 + \mu_1) + (\lambda_1 - \mu_1)H + (\lambda_2 - \mu_1)P + \mu_1 G+\\
	&(\lambda_3 - \lambda_1)\{h^\ell,h^{n-\ell}\}+ (\mu_2-\mu_1)(\{h^\ell ,h^{n-\ell}\} \times \{g,g^2,\ldots,g^{m-1}\})
\end{aligned}
\ee			
holds in $\Z[G]$.

\begin{remark} Let $G = H \times P$, $H= \langle h \rangle$ and $P= \langle g \rangle $ of order $n$ and $m$. An 1-$(n,m,k,\lambda_1,\lambda_2,\lambda_1,\mu,\mu)$-PDPDS in $G$ relative to $H$ and $P$ is an $(n,m,k,\lambda_1,\lambda_2,\mu)$-DPDS in $G$ relative to $H$ and $P$. Moreover,  an 1-$(n,m,k,\lambda,0,\lambda,\lambda,\lambda)$-PDPDS in $G$ relative to $H$ and $P$ is an $(n,m,k,\lambda)$-RDS in $G$ relative to $P$. Finally, an 1-$(n,m,k,\lambda,\lambda,\lambda,\lambda,\lambda)$-PDPDS in $G$ relative to $H$ and $P$ is a cyclic $(nm,k,\lambda)$-DS in $G$. Throughout the paper,  the notation $(n,m,k,\lambda_1,\lambda_2,\lambda_3,\mu_1,\mu_2)$-PDPDS will denote 1-$(n,m,k,\lambda_1,\lambda_2,\lambda_3,\mu_1,\mu_2)$-PDPDS.
\end{remark}

The construction methods of difference sets are widely studied in the literature (see \cite{beth1999design,colbourn2006handbook}).
Below, we give a construction method of a family of $\ell$-PDPDS.

\begin{proposition}\label{prop:lpdpds:cons}
Let $G = H \times H$ be a set, $H= \langle h \rangle$ be a cyclic group of order $n$, and $a,b,\ell\in \Z^+ \cup \{0\}$ such that $|a-b|=\ell$. Let  $R \subset G$ be a set defined as $$R=\{h^a\times H \backslash \{h^b\}\} \cup \{H\backslash \{h^b\} \times h^b\}.$$
If $\ell\neq 0$, then $R$ is an $\ell$-$(n,n,2n-2,n-2,n,n-2,2,1)$-PDPDS in $G$ relative to $H$. If $\ell=0$, then $R$ is an $(n,n,2n-2,n-2,n-2,2)$-DPDS in $G$ relative to $H$. 
\end{proposition}
\begin{proof}
    Let $A$ and $B$ be the sets defined as $h^a\times H \backslash \{h^b\}$, $H\backslash \{h^b\} \times h^b$, respectively.  Therefore we can write $$RR^{-1}=AA^{-1}\cup AB^{-1} \cup BB^{-1} \cup BA^{-1}.$$
    Now we examine this difference set one by one. We first consider the case $\ell\neq 0$. If $a-b=\ell$, then
    $$AA^{-1}=\{\underbrace{(0,0),}_{n-1 \text{ times }}\underbrace{(0,1)}_{n-2\text{ times }}, \ldots, \underbrace{(0,n-1)}_{n-2\text{ times }}\},$$
    $$AB^{-1}=\Z_n\backslash\{\ell\}\times\Z_n^*$$
    $$BB^{-1}=\{\underbrace{(0,0),}_{n-1 \text{ times }}\underbrace{(1,0)}_{n-2\text{ times }},\ldots ,\underbrace{(n-1,0)}_{n-2\text{ times }}\},$$
    $$BA^{-1}=\Z_n\backslash\{n-\ell\}\times\Z_n^*.$$
      Hence, by Definition \ref{df:l-pDPDS}, $R$ is an $\ell$-$(n,n,2n-2,n-2,n,n-2,2,1)$-PDPDS in $G$ relative to $H$. If  $a-b=-\ell$, the sets $AB^{-1}$ and $BA^{-1}$ given above are replaced with each other, and so we get the result. 
     
     Similarly, for the case $\ell=0$, we can see that $R$ is an $(n,n,2n-2,n-2,n-2,2)$-DPDS in $G$ relative to $H$.
\end{proof}
\begin{example}
    Let $G=\Z_7\times \Z_7$, $a=5$ and $b=3$. Then we get
    $$R=\{(5,0),(5,1),(5,2),(5,4),(5,5),(5,6),(0,3),(1,3),(2,3),(4,3),(5,3),(6,3)\}$$ which is an $2$-$(7,7,12,5,7,5,2,1)$-PDPDS in $G$ relative to $\Z_7$. Next we consider $G=\Z_5\times \Z_5$ and $a=b=2$. Then we get
    $$R=\{(2,0),(2,1),(2,3),(2,4),(0,2),(1,2),(3,2),(4,2)\}$$ which is an $(5,5,8,3,3,2)$-DPDS in $G$ relative to $\Z_5$.
\end{example}

We now draw the connection between the definition of $\ell$-PDPDS and the perfect sequences. Let $m$ be a prime number, $n \geq 2$ be an integer, and $\underline{a} = (a_0,a_1,\ldots,a_{n},\ldots)$ be an almost $m$-ary sequence of period $n+s$ with $s$ zero-symbol such that $a_{i_j} = 0$ for all $j=1,2,\ldots , s$ where $\{ i_1,i_2, \ldots , i_s \} \subset \{0,1, \ldots , n+s-1\}$.
	Let $H = \langle h \rangle$ and $P = \langle g \rangle $ be the (multiplicatively written) cyclic groups of order $n+s$ and $m$, respectively. Let $G$ be the group defined as $G=H \times P$.
	We choose a primitive $m$-th root of unity, $\zeta_m \in \C$. For $i \in  \{0,1, \ldots , n+s-1\} \backslash \{ i_1,i_2, \ldots , i_s \}$ let $b_i$ be the integer in $\{0,1,2,\ldots,m-1\}$ such that $a_i = \zeta_m^{b_i}$. 
	Let $R_a$ be a subset of $G$ defined as
	
	\be \label{eqn:NPS2DPDS}
	R_a= \{(g^{b_i}h^i) \in G : i \in  \{0,1, \ldots , n+s-1\} \backslash \{ i_1,i_2, \ldots , i_s \}\}.
	\ee
From now on, we will use the notation $R_a$ for a difference set if it is related with a sequence. On the other hand, we will use $R$ if we are talking about a general difference set.

The authors in \cite{ozdenyayla2020} studied almost $m$-ary NPS $\underline{a}$ with two consecutive zero-symbols and obtained that $R_a$ is an PDPDS. Then, some non-existence results for NPS with two consecutive zero-symbols were proven by using their PDPDS equivalence. 
Here, we extend this one-to-one equivalence to the non-consecutive case and show that  an NPS with two zero-symbols is related to an $\ell$-PDPDS, see the consecutive zero result in \cite[Theorem 3]{ozdenyayla2020}. As its proof is very similar to the consecutive case, we do not give its proof here.
\begin{proposition}
 \label{th:l_PDPDS}
	Let $\gamma_1,\gamma_2$ be integers. Then
	$\underline{a}$ is an almost $m$-ary NPS of type $(\gamma_1, \gamma_2)$ with two zero-symbols at locations 0 and $\ell$ if and only if $R_a$ is an $\ell$-$(n+2,m,n,\frac{n-\gamma_2 - 2}{m}+\gamma_2,0,\frac{n-\gamma_1 -1}{m}+\gamma_1,\frac{n-\gamma_2 - 2}{m},\frac{n-\gamma_1 -1}{m})$-PDPDS.
\end{proposition}
\begin{remark}
	According to Proposition \ref{th:l_PDPDS}, we see that an $m$-ary NPS of type $(\gamma_1, \gamma_2)$ with two zero-symbols is equivalently defined by an $\ell$-PDPDS. Herein, we note that 
	\be \nn
	C_{\underline{a}}(t) = \left\lbrace \begin{array}{ll} \gamma_1& \mbox{if } \mbox{$t= \ell, n- \ell$},\\ \gamma_2 & \mbox{if } \mbox{ otherwise}. \end{array} \right.
	\ee
\end{remark}
We give below some examples of Proposition \ref{th:l_PDPDS} for sequences with nonconsecutive zero-symbols.
\begin{example}
	The sequence $\underline{a}=(0,\zeta_3,\zeta_3,\zeta_3,0,\zeta_3,\zeta_3^2,\zeta_3^2,\zeta_3,\ldots)$  is an almost ternary NPS of type $(0,2)$ with two zero-symbols and $R_a$ is a 
	4-(9,3,7,3,0,2,1,2) PDPDS in $\Z_9\times\Z_3$.
Similarly, 
	the sequence $\underline{a}=(0,1,1,0,\zeta_2,1,\zeta_2,\zeta_2,1,\zeta_2,\ldots)$  is an almost binary NPS of type $(3,-2)$ with two zero-symbols and $R_a$ is a 
	3-(10,2,8,4,0,5,4,2) PDPDS in $\Z_{10}\times\Z_2$.
\end{example}
By using Definition \ref{df:l-pDPDS}, we present  some results on  the parameters of an $\ell$-PDPDS,  which we give in Proposition \ref{prop:s_i}. We note that this is a direct extension of  \cite[Proposition 4]{ozdenyayla2020}, and so its proof follows similarly. 

\begin{proposition}\label{prop:s_i}
Let $ R $ be an $\ell$-$(n,m,k,\lambda_1,\lambda_2,\lambda_3,\mu_1,\mu_2)$-PDPDS in $G$ relative to $H$ and $P$. Let $s_i \in \Z^+\cup \{0\}$ be defined as $s_i = |\{(r_1,r_2) \in R : r_2=i\}|$. 
		Then 
		\begin{itemize}
		    \item[(i)] $ 
		\sum_{j=0}^{m-1}{s_j}^2=(n-3)\lambda_1+2\lambda_3+k
		$
		\item[(ii)]  $ \sum_{j=0}^{m-1}s_j s_{j-i}=(n-3)\mu_1+\lambda_2+2\mu_2
		$
		\item[(iii)] $(m-1)\sum_{j=0}^{m-1}s_j s_{j-i}+\sum_{j=0}^{m-1}{s_j}^2=k^2$
		\end{itemize}
		for each $ i=1,2,\dots, \lceil\frac{m-1}{2}\rceil$, where the subscripts are computed modulo $ m $. 
\end{proposition}

\begin{example}
    Let $R=\{(5,0),(5,1),(5,2),(5,4),(5,5),(5,6),(0,3),(1,3),(2,3),(4,3),(5,3),(6,3)\}$ where $s_i=1$ for $i\in \{0,1,2,4,5,6\}$ and $s_3=6$. Then  $R$ is an 2-(7,7,12,5,7,5,2,1) PDPDS in $G$ relative to $\Z_7$. Note that  $R$ satisfies (i)-(ii)-(iii) in Proposition \ref{prop:s_i}. 
    
\end{example}
Next, by using Proposition \ref{th:l_PDPDS} and Proposition \ref{prop:s_i} we get in Proposition \ref{prop:non} a non-existence result of an $m$-ary NPS of type $(\gamma_1,\gamma_2)$ with two zero-symbols for $\gamma_2 \le -3$, see \cite[Theorem 4]{ozdenyayla2020}.
\begin{proposition} \label{prop:non}
Let $\gamma_1,\gamma_2$ be integers.
Let $ m $ be an odd prime, 
		let $n-\gamma_2 - 2 = k_1m$ and $n-\gamma_1-1 =k_2m$ for some $k_1,k_2 \in \N$. Then, there does not exist 
		an almost $ m $-ary sequence of  type $ (\gamma_1,\gamma_2) $ and period $ n+2 $ with two zero-symbols 
		for $\gamma_2 \le \left\lfloor\frac{-mk_1-4+ \sqrt{m^2k_1^2-4mk_1+8mk_2}}{2}\right\rfloor$, and so 
		for  $\gamma_2 \leq -3$.
\end{proposition}

We now consider the multipliers of the set $R_a$.
We give first the definition of multiplier of a general difference set. Let $ R $ be an $\ell$-PDPDS in $ G $ relative to $ H $ and $ P $.  The set $R^{(t)}$ is defined by $$R^{(t)}=\{tr:r\in R\}\subset G $$  for $ t \in \Z $. If there exist $ g\in G $ such that $ R^{(t)}=R+g\subset G $ for some $t \in \Z$ such that $ gcd(t,|G|)=1$,  then we call that $ t $ is a \textit{multiplier} of $ R $.
	The following result gives us a necessary condition on  the multipliers of the set  $R_a$.
\begin{proposition}\label{prop:multiplier}
Let  $R_a \subset G$ be defined as in $\eqref{eqn:NPS2DPDS}$ for $s=2$ such that $i_1=0$ and $i_2=\ell$. 
If $ t $ is a multiplier of $ R_a $, then $ t \equiv \pm 1 \mod (n+2) $.
\end{proposition}
\begin{proof}
Let $ R_a=\{(1,b_1),(2,b_2),\dots,(\ell-1,b_{\ell-1}),(\ell+1,b_{\ell+1}), \ldots,(n+1,b_{n+1})\} $ and $ R_1 $ denote the set of first components of the elements in $ R_a $, i.e. $ R_1=\{1,2,\dots,\ell-1,\ell+1,\ldots,n+1\} $. Similarly, we define $R^{(t)}_1 $ and $G_1$  as $ R^{(t)}_1=\{tr_1: r_1 \in R_1\}$  and $G_1=\{0,1,\ldots,n+1\}$, respectively. We suppose that $ t $ is a multiplier of $ R_a $. Thus, there exists $g_1 \in G_1$ such that $ R^{(t)}_1 = R_1+g_1$. Assume that $ g_1\in G_1\backslash\{0,n+2-\ell\}$, then  there exists $r_1 \in R_1$ such that $ r_1+g_1=0 $. However, $0 \not\in R^{(t)}_1 $ as $ \gcd(t,|G|)=1 $. So, $ g_1  \in \{0,n+2-\ell\} $. 
Firstly, if $ g_1=n+2-\ell $, then
\be \label{eq:mult_g}
R_1+g_1=\Z_{n+2}\backslash\{0,-\ell\}.
\ee 
On the other hand, as $ \gcd(t,|G|)=1 $ and  $R_1=\Z_{n+2}\backslash\{0,\ell\}$, we have 
\be \label{eq:mult_t}
R^{(t)}_1=\Z_{n+2}\backslash\{0,t\ell\}
\ee
By \eqref{eq:mult_g} and \eqref{eq:mult_t}, we get $ t  \equiv -1 \mod (n+2) $. If $ g_1=0 $, then we get $ R_1+g_1=\{1,2,\dots,\ell-1,\ell+1,\ldots,n+1\} $. It is now clear that $ t \equiv 1 \mod (n+2) $.
\end{proof}
 \\
 
In the remaining part of this section, we consider the case $\ell=1$.
\begin{theorem} \label{thm:sym}
		Let $m$ be an odd prime, $R_a \in G=H \times P$ be defined as in \eqref{eqn:NPS2DPDS} for $s=2$ and $\ell = 1$. If $R_a$ is a $1$-$(n,m,k,\lambda_1,\lambda_2,\lambda_3,\mu_1,\mu_2)$-PDPDS in $G$ relative to $H$ and $P$, then $b_i=b_{n+3-i}$.
		Moreover, if $n$ is an odd (even) integer, then $\lambda_1$ is a positive odd (resp. even) integer. 

\end{theorem}
\begin{proof}
As $R_a=\{(2,b_2),(3,b_3),\dots,(n+1,b_{n+1})\}$ we have 
\be \label{eq:symm}
\begin{array}{ll}		    	R_a R_a^{(-1)} - n(0,0)=&\{(1,b_3-b_2),(1,b_4-b_3),\dots,(1,b_{n+1}-b_n),\\ &(2,b_4-b_2),(2,b_5-b_3),\dots,(2,b_{n+1}-b_{n-1}),\\ 
	&\cdots, \\
	&(n+1,b_2-b_3),(n+1,b_3-b_4),\dots,(n+1,b_n-b_{n+1})\} .
\end{array}
\ee
Since $R_a$ is a PDPDS, the first and last  rows (other rows) in \eqref{eq:symm}  cover the set $P\backslash e_P$ exactly $\mu_1$ (resp. $\mu_2$) times. Thus  we get
\be
\begin{array} {ccc}
  (b_3-b_2)+(b_4-b_3)+\dots+(b_{n+1}-b_n) &\equiv & 0 \mod m  \\
  (b_4-b_2)+(b_5-b_3)+\dots+(b_{n+1}-b_{n-1}) &\equiv& 0 \mod m \\
  &\vdots & \\
  (b_2-b_3)+(b_3-b_4)+\dots+(b_n-b_{n+1}) &\equiv & 0 \mod m \nn
\end{array}
\ee
Therefore, we get  $b_i=b_{n+3-i}$.

For the second part, 
we know that $\lambda_1+(m-1)\mu_1=n-2$ and $\lambda_3+(m-1)\mu_3=n-1$ by the definition of PDPDS. If $n$ is odd (even), then $\lambda_1$ is odd (resp. even) and $\lambda_3$ is  even (resp. odd).  
	Besides, when $n$ is a odd, $\lambda_1 \geq 1$ as $\{(2,0),(3,0),\dots,(n,0)\} \subset R_a R_a^{-1}$ by the first part. Similarly, if $n$ is  even, then $\lambda_1 \geq 2$, $\lambda_3 \geq 1$ as $\{(3,0),(5,0),\dots,(n-1,0),(3,0),(5,0),\dots,(n-1,0),(1,0),(n+1,0)\} \subset R_a R_a^{-1}$. Therefore, we get the  result.
\end{proof} 
	
 Theorem \ref{thm:sym} says that an almost $m$-ary NPS with two consecutive zero-symbols for an odd prime $m$ is symmetric, which we give in the following corollary. Here we take the first zero-symbol in each period to the end of the period of the sequence. As the sequence is periodic, this does not affect the auto-correlation coefficients. 
	
\begin{corollary} \label{cor:sym}
	Let $m$ be an odd prime, $\gamma_1,\gamma_2 \in \Z$ and and $n \in \Z^+$. An almost $m$-ary NPS of type $(\gamma_1, \gamma_2)$ and period $n+2$ with two consecutive zero-symbols is symmetric. 
\end{corollary}
\begin{example}
	$\underline{a}_1=(0,\zeta_3^2,1,\zeta_3^2,0,\ldots)$ and $\underline{a}_2=(0,\zeta_3^2,\zeta_3,1,\zeta_3,\zeta_3^2,0,\ldots)$  are almost ternary NPS of type $(\gamma_1,\gamma_2)$ with two consecutive zero-symbols and their first periods are symmetric with respect to 1 at their center. 
	The converse of Corollary \ref{cor:sym} is not correct. 
	For instance, $\underline{a}_1=(0,\zeta_3,\zeta_3^2,\zeta_3,\zeta_3^2,\zeta_3,0,\ldots)$, $\underline{a}_2=(0,\zeta_3^2,\zeta_3^2,1,\zeta_3^2,\zeta_3^2,0,\ldots)$ and $\underline{a}_3=(0,\zeta_3^2,1,\zeta_3^2,1,\zeta_3^2,0,\ldots)$ are almost ternary symmetric sequences of period $7$ with two consecutive zero-symbols, 
	but these sequences are not a NPS. 
\end{example}

By using the multipliers given in Proposition \ref{prop:multiplier}  and the symmetry property given in Theorem \ref{thm:sym}, we will study the set $R_a$ to get some results on the almost $m$-ary NPS. 
It is well known that one can construct a difference set  by taking union of some orbits of a multiplier (see \cite[Theorem 1.3.8]{pott2006finite} and \cite[Result 6]{chee2010almost}). Furthermore,  one can decide the existence and the non-existence of difference sets.
Now we consider the difference set $R_a$ defined in \eqref{eqn:NPS2DPDS} and we try to write it a union of some orbits of its multiplier.

We first consider the sequences with consecutive zero-symbols and then we deal with the non-consecutive case. Let $m$ be an odd prime, $R_a \subset G$ be defined as in $\eqref{eqn:NPS2DPDS}$ for $s=2$ such that $i_1=0$ and $i_2=1$ and $t$ be a multiplier of $R_a$. 

Suppose that $t \equiv 1 \mod(n+2)$ and $t \equiv k \mod m$. Hence, we have $R_a^{(t)} = R_a + g$ where $g=(0,g_2)\in G$ for some $g_2\in \Z_m$. Then we get $R_a^{(t)}=\{(2,kb_2),(3,kb_3),\ldots,(n+1,kb_{n+1}\}$ and $R_a+g=\{(2,b_2+g_2),(3,b_3+g_2),\ldots,(n+1,b_{n+1}+g_2)\}$.  Therefore, we get $b_2=b_3=\cdots=b_{n+1}$, i.e. $\underline{a}$ is trivial, unless  $k = 1$. 

Now, suppose that $t \equiv -1 \mod(n+2)$ and $t \equiv k \mod m$. We have $R_a^{(t)} = R_a + g$ where $g=(n+1,g_2)\in G$ for some $g_2\in \Z_m$. Then we consider the sets $R_a^{(t)}=\{(n,kb_2),(n-1,kb_3),\ldots,(1,kb_{n+1}\}$ and $R_a+g=\{(1,b_2+g_2),(2,b_3+g_2),\ldots,(n,b_{n+1}+g_2)\}$. As the second components of the elements in $R_a$ are symmetric by Theorem \ref{thm:sym},  we similarly get that this sequence is trivial unless $k=1$.  

We finally consider the case $t\equiv\pm1\mod (n+2)$ and $t\equiv1 \mod m$. For the case $t\equiv -1\mod (n+2)$ and $t\equiv1 \mod m$, one can not have a symmetric sequence.
Besides, for the case $t\equiv 1\mod (n+2)$ and $t\equiv1 \mod m$ all orbits are of length 1.  Therefore, we have proved the following result.
\begin{proposition} \label{prop:orbit}
 Let $m$ be an odd prime. Let $\underline{a}$ be a non-trivial almost $m$-ary NPS of type $(\gamma_1,\gamma_2)$  with consecutive two zero-symbols. Let $R_a \in G$ be a set constructed as in \eqref{eqn:NPS2DPDS} and $t$ be a multiplier of $R_a$.  Let $\Phi$ be any collection of the orbits of $G$ under $t$. Then $$R_a \neq \bigcup_{A\in \Phi}A.$$
\end{proposition}

On the other hand, when we consider the sequences $\underline{a}$ with non-consecutive two zero-symbols, we get that one can write the set $R_a \in G$ as a union of some orbits of $G$ under a multiplier $t$ of $R_a$. We give its example below.
\begin{example} \label{ex:orbit}
Let  $R_a \subset G$ be defined as in $\eqref{eqn:NPS2DPDS}$ for $s=2$ such that $i_1=0$ and $i_2\neq 1$ and $t$ be a multiplier of $R_a$. 
Let $G=\Z_{10}\times \Z_3$  and we choose $t=19\equiv -1 \mod 10$ as in  Proposition \ref{prop:multiplier}. 
The orbits of $G$ under the action $x\rightarrow 19x$ are
    \{ (0, 2) \},
  \{ (5, 0) \},
  \{ (4, 0), (6, 0) \},
  \{ (1, 1), (9, 1) \},
  \{ (7, 0), (3, 0) \},
  \{ (0, 1) \},
  \{ (9, 0), (1, 0) \},
  \{ (5, 2) \},
 \{ (0, 0) \},
 \{ (8, 2), (2, 2) \},
 \{ (5, 1) \},
 \{ (7, 2), (3, 2) \},
 \{ (7, 1), (3, 1) \},
 \{ (1, 2), (9, 2) \},
 \{ (8, 1), (2, 1) \},
 \{ (6, 2), (4, 2) \},
 \{ (6, 1), (4, 1) \},
  \{ (2, 0), (8, 0) \}.
  Then we, by exhaustive search, get that  $R_a= \{ (1, 2), (9, 2) \} \cup \{ (8, 2), (2, 2) \} \cup  \{ (7, 0), (3, 0) \} \cup \{ (6, 1), (4, 1) \}$ which is a $5$-(10,3,8,2,0,0,2,4)-PDPDS. Herein, the  sequence $\underline{a}=(0,\zeta_3^2,\zeta_3^2,1,\zeta_3,0,\zeta_3,1,\zeta_3^2,\zeta_3^2 ,\ldots)$  is an almost ternary NPS of type $(-4,0)$ with two zero-symbols. 
\end{example}

We close this section with a conjecture about the non-existence of an almost $m$-ary NPS. 
\begin{conjecture}\label{conj:1}
Let $m\geq5$ be a prime, $\gamma_1,\gamma_2 \in \Z$ and $n \in \Z^+$.
There does not exist a non-trivial almost $m$-ary NPS of type $(\gamma_1, \gamma_2)$ and period $n+2$ with two consecutive zero-symbols.
\end{conjecture}

\section{Construction}\label{sec:cons}
In this section, we present two construction methods for a family of almost $m$-ary sequences. First, we show that the auto-correlation coefficients of an $m$-ary sequence is invariant under some operations. 
\begin{theorem} \label{thm:p_ary-const.}
    Let $m$ be a prime and \underline{a} be an $m$-ary sequence of period $n$ such that $C_{\underline{a}}(t) \in \Z$ for $t=1,2,\ldots, n-1 $. Let the sequences 
    $\underline{b}$, $\underline{c}$ be defined by 
    $\underline{b}=\zeta_m^k\underline{a}$, and $\underline{c}=\underline{a}*\underline{b}$, where $k \in \Z_m$ and $*$ is a component-wise multiplication. Then, $$C_{\underline{a}}(t)=C_{\underline{b}}(t)=C_{\underline{c}}(t)$$
    for $t=1,2,\ldots, n-1 $.
\end{theorem}
\begin{proof}
By using \eqref{df:autocor}, 
we have 
\be  \nn 
\begin{array}{ll}
C_{\underline{b}}(t)=\dd \sum_{i=0}^{n-1}{b_i\overline{b_{i+t}}}
=\dd \sum_{i=0}^{n-1}{a_i\zeta_m^k\overline{a_{i+t}\zeta_m^k}} 
=\dd \sum_{i=0}^{n-1}{a_i\zeta_m^k\overline{a_{i+t}\zeta_m^k}} 
=C_{\underline{a}}(t),
\end{array}
\ee
and
\be \label{eq:C_c}
\begin{array}{ll}
C_{\underline{c}}(t)=\dd\sum_{i=0}^{n-1}{c_i\overline{c_{i+t}}}
=\dd \sum_{i=0}^{n-1}{a_ib_i\overline{a_{i+t}b_{i+t}}} 
=\dd \sum_{i=0}^{n-1}{a_ia_i\zeta_m^k\overline{a_{i+t}a_{i+t}\zeta_p^k}} 
= \sum_{i=0}^{n-1}{a_i^2\overline{a_{i+t}^2}}.
\end{array}
\ee
Since $m$ is a prime and $C_{\underline{a}}(t) \in \Z$ for $t=1,2,\ldots, n-1 $, there exist $x,y \in \Z$ such that
\be \label{eq:C_a}
C_{\underline{a}}(t) =  \sum_{i=0}^{n-1}{a_i\overline{a_{i+t}}} =x+ y\sum_{i=1}^{m-1}\zeta_m^i.
\ee
Then, 
\be \label{eq:C_a2}
\sum_{i=0}^{n-1}{a_i^2\overline{a_{i+t}^2}} = x+ y\sum_{i=1}^{m-1}\zeta_m^{2i} = x+ y\sum_{j=1}^{m-1}\zeta_m^j.
\ee
Hence, we get the result by combining \eqref{eq:C_c}, \eqref{eq:C_a}  and \eqref{eq:C_a2}.
\end{proof}

\begin{example}
Let  $\underline{a}=(0,0,\zeta_3^2,\zeta_3,1,\zeta_3,\zeta_3^2,\ldots)$. By Theorem \ref{thm:p_ary-const.}, we get
\be \nn
\begin{array}{ll}
   \underline{a}_1=\underline{a}=(0,0,\zeta_3^2,\zeta_3,1,\zeta_3,\zeta_3^2,\ldots) \\ 
    \underline{a}_2=\zeta_3^1*\underline{a}_1=(0,0,1,\zeta_3^2,\zeta_3,\zeta_3^2,1,\ldots)\\
    \underline{a}_3=\zeta_3^2*\underline{a}_1=(0,0,\zeta_3,1,\zeta_3^2,1,\zeta_3,\ldots)\\
    \underline{a}_4=\underline{a}_1*\underline{a}_2=(0,0,\zeta_3^2,1,\zeta_3,1,\zeta_3^2,\ldots)\\
     \underline{a}_5=\zeta_3^1*\underline{a}_4=(0,0,1,\zeta_3,\zeta_3^2,\zeta_3,1,\ldots) \\
     \underline{a}_6=\zeta_3^2*\underline{a}_4=(0,0,\zeta_3,\zeta_3^2,1,\zeta_3^2,\zeta_3,\ldots).
\end{array}
\ee
 These sequences are almost ternary NPS of type $(-2,0)$ with two consecutive zero-symbols.
\end{example}
\begin{remark}
By Theorem \ref{thm:p_ary-const.}, we can construct $2m$ many $m$-ary sequences from an $m$-ary sequence with the same autocorrelation values. We note that Theorem \ref{thm:p_ary-const.} does not hold  for a composite number $m$. Let $\underline{a}$, $\underline{b}$, $\underline{c}$ be sequences defined as in Theorem \ref{thm:p_ary-const.}. The sequence $\underline{c}$  can  have a different autocorrelation value than $\underline{a}$ and $\underline{b}$. 
For example, we take an almost quaternary sequence  $\underline{a}=(0,0,1,\zeta_4,\zeta_4^2,\zeta_4^2,\zeta_4^2,\zeta_4^3,\zeta_4^2,\zeta_4,\zeta_4^2,\zeta_4^3,\zeta_4^2,\zeta_4^2,\zeta_4^2,\zeta_4,1,\ldots)$. and consider the following sequences
\be \nn
\begin{array}{ll}
   \underline{a}=(0,0,1,\zeta_4,\zeta_4^2,\zeta_4^2,\zeta_4^2,\zeta_4^3,\zeta_4^2,\zeta_4,\zeta_4^2,\zeta_4^3,\zeta_4^2,\zeta_4^2,\zeta_4^2,\zeta_4,1,\ldots), \\ 
    \underline{b}=\zeta_4^1*\underline{a}=(0,0,\zeta_4,\zeta_4^2,\zeta_4^3,\zeta_4^3,\zeta_4^3,1,\zeta_4^3,\zeta_4^2,\zeta_4^3,1,\zeta_4^3,\zeta_4^3,\zeta_4^3,\zeta_4^2,\zeta_4,\ldots),\\
    \underline{c}=\underline{a}*\underline{b}=(0,0,\zeta_4,\zeta_4^3,\zeta_4,\zeta_4,\zeta_4,\zeta_4^3,\zeta_4,\zeta_4^3,\zeta_4,\zeta_4^3,\zeta_4,\zeta_4,\zeta_4,\zeta_4^3,\zeta_4,\ldots).
\end{array}
\ee The sequences $\underline{a}$ and $\underline{b}$  are almost quaternary NPS of type $(4,1)$ with two consecutive zero-symbols. However, the sequence $\underline{c}$ has three distinct autocorrelation values: $-6$, $5$, and $-3$.
\end{remark}

Now, we give some preliminaries on cyclotomic classes, which we use to construct some almost $m$-ary nearly perfect sequences.
\begin{definition} \label{df:cyc-cst}
    Let $q$ be a prime such that $q=mf+1$ where $m,f\in \Z^+$ and $\alpha$ be a primitive element of $\Z_q^*$. The $m^{th}$ cyclotomic classes of $\Z_q$ are defined as
    $$D_k:=\{\alpha^{mi+k}| i=0,1,\dots, f-1\},$$
    for $k=0,1,\dots,m-1$. The cyclotomic classes disjointly decomposes $\Z_q^*$ as  $\Z^*_q= {\dd\bigsqcup^{m-1}_{k=0} D_k}$.
\end{definition}

We give the definition of an almost $m$-ary $\sigma$-sequence as given in \cite{shi2019family} and then we present a theorem about the almost $m$-ary $\sigma$-sequences.
\begin{definition}\cite[Definition 2]{shi2019family}\label{df:m-ary sigma seq.1}
    Let $\sigma$ be a permutation of $\Z_m$. The sequence $\underline{a}=(a_0,a_1,\ldots,a_{q-1},\ldots)$ is called an almost $m$-ary $\sigma$-sequence if 
    	\be \nn
		a_i= \left\lbrace \begin{array}{ll}  \zeta_m^k &  \mbox{if } i \in D_{\sigma(k)},\\ 0 & \mbox{i=0.} 
		\end{array} \right.
        \ee
\end{definition}
\begin{theorem}\cite[Theorem 3]{shi2019family}\label{thm:m_ary-const11}
   Let $q$ be a prime such that $q=mf+1$ for some $m,f\in\Z^+$. Let $\underline{a}$ be an almost $m$-ary $\sigma$-sequence of period $q$ for a linear permutation $\sigma$. Then the sequence $\underline{a}$ is an almost $m$-ary NPS of type $-1$ with one zero-symbol.
\end{theorem}

We extend Theorem \ref{thm:m_ary-const11} for the case $q=2p=mf+l$ such that $l=p+1$ below.
\begin{proposition} \label{thm:m_ary-const1}
    Let $p$ be a prime, and $q=2p=mf+l$ such that $m,f\in\Z^+$ and $l=p+1$. Let $\underline{a}$ be an almost $m$-ary $\sigma$-sequence of period $  q$ for some linear permutation $\sigma$. 
    Then the sequence $\underline{a}$ is  an almost m-ary NPS of type $(\gamma_1,\gamma_2)$ with $p+1$ zero-symbols. In particular, 
    	\be \nn
		C_{\underline{a}}(t)= \left\lbrace \begin{array}{ll}  -1 &  \mbox{if } t\equiv 0 \mod 2,\\ 0 & \mbox{if } t\equiv 1 \mod 2.
		\end{array} \right.
        \ee
\end{proposition}
\begin{proof}
    As $\underline{a}$ is an almost $m$-ary $\sigma$-sequence, we have $a_i=0$ for $i\equiv 0 \mod 2$ and $i=p$ by Definition \ref{df:m-ary sigma seq.1}. If $t\equiv 1 \mod 2$, then it is clear that $C_{\underline{a}}(t)=0$.
On the other hand, when $t \not\equiv 1 \mod 2$, the sequence $\underline{a}$ turns into the sequence defined in Definition \ref{df:m-ary sigma seq.1}. Therefore, 
the proof follows by Theorem \ref{thm:m_ary-const11}.
\end{proof}
\begin{example} \label{ex:D_k}
Let $q=14$, $m=3$, $l=8$ and $\alpha=5$. We take $\sigma$  the identity permutation, i.e.~$\sigma(l)\equiv l \mod 3$. Then,
\be \nn
\begin{array}{ll} 
D_0=\{1,5^3\}=\{1,13\}, \\
D_1=\{5,5^4\}=\{5,9\}, \\
D_2=\{5^2,5^5\}=\{11,3\}. \\
\end{array}
\ee Here we can see that  $\Z^*_{14}=\bigcup^{2}_{k=0} D_k$. The sequence $\underline{a}=(0,1,0,\zeta^2_3, 0,\zeta_3,0,0,0,\zeta_3,0,\zeta^2_3,0,1,\dots)$ is an almost ternary NPS of type $(0,-1)$ with eight zero-symbols.
\end{example}

\begin{lemma}\label{lem:sum-D_k}
    Let $q$ be a prime such that $q=mf+1$ for some $m,f\in\Z^+$. Let  $\alpha$ be a generator of $\Z_q^*$. Then the sum of the elements of $D_k$ is $0$, i.e.
    \be \nn
    \left(\sum^{f-1}_{i=0}\alpha^{mi+k}\right) \mod q \equiv 0
    \ee
    for $k =0,1,\ldots, m-1$. Moreover, if $f$ is even, then
$$\alpha^{mi+k}+ \alpha^{m(i+\frac{f}{2})+k}\equiv 0 \mod q$$ for $i=0,1,\ldots,\frac{f}{2}-1.$
\end{lemma}
\begin{proof}
We can write the sum of elements in $D_k$  as
 \be \nn
 \dd 
     \sum^{f-1}_{i=0}\alpha^{mi+k}  \equiv\alpha^k \left(\dd\sum^{f-1}_{i=0}\alpha^{mi}\right)
      \equiv\alpha^k\left(\frac{1- \alpha^{mf}}{1-\alpha^m}\right)
    \ee
    for $k =0,1,\ldots, m-1$. Since $\alpha^{mf}\equiv 1 \mod q$, the proof of the first part is completed. Suppose that $f$ is even, then the equation array
    \be \nn
        \alpha^{mi+k}+\alpha^{m(i+\frac{f}{2})+k}
        \equiv\alpha^{mi+k}(1+\alpha^{\frac{mf}{2}}) 
        \equiv 0 \mod q
    \ee
    proves the second part.
\end{proof}

\begin{remark}
We note that Lemma \ref{lem:sum-D_k} also holds for an integer $q = 2p = mf +l$ such that $m,f\in\Z^+$ and $l = p-1$, see Example \ref{ex:D_k}.
\end{remark}
Lemma \ref{lem:sum-D_k} gives us a necessary condition on the existence of an almost $m$-ary $\sigma$-sequence. We state this in the following proposition. 
\begin{proposition} \label{prop:sym} Let $\sigma$ be a linear permutation of $\Z_m$. If either of the following holds
    \begin{enumerate}
        \item[i.] $q$ is a prime such that $q=mf+1$ for some $m\in\Z^+$ and some even integer $f$, \item[ii.]  $q= 2p = mf+l$ for some $m\in\Z^+$, some even integer $f$ and $l=p-1$,
    \end{enumerate}
    then the almost $m$-ary $\sigma$-sequence is symmetric except for $a_0$.
\end{proposition}
\begin{proof}
By using the second part of Lemma \ref{prop:sym} and that $\sigma$ is linear, we have $q-i \in D_{\sigma(k)}$ whenever $i \in D_{\sigma(k)}$ for any $k$. Therefore, $a_i = \zeta^k = a_{q-i}$.
\end{proof}

   For example, take the sequence $\underline{a}=(0,1,0,\zeta^2_3,0,\zeta_3,0,0,0,\zeta_3,0,\zeta^2_3,0,1,\ldots)$ which is symmetric except for $a_0$.

Now, we  present a construction of a family of quaternary nearly perfect sequences of period $q$ of type $\frac{q-3}{2}$. First of all, we present a few well-known properties of cyclotomic numbers.
We define the cyclotomic numbers $(k,l)$ as
\be \label{eq:cyc_nbr}
(k,l)=|(D_k+1)\cap D_l|, 
\ee
for $0\leq k, l \leq f-1$, where $D_k$ is the $m^{th}$ cyclotomic class of $\Z_q$ defined in Definition \ref{df:cyc-cst}.

\begin{lemma} \cite{dickson1935cyclotomy} \label{prop:cyc_nbr}
Let $q=mf+1$ be a prime for some $m,f \in \Z^+$ and $(k,l)$ be a cyclotomic number for some $0\leq k, l \leq f-1$. Then,
\begin{itemize}
    \item[(i)]     $ 
	(k,l)= \left\lbrace 
	\begin{array}{ll}  
	(m-k,l-k)&  \\ 
	(l,k)& \mbox{ f even }  \\
	(l+\frac{m}{2},k+\frac{m}{2})& \mbox{ f odd.}
	\end{array} \right.
$
\item[(ii)] 
$
	\dd \sum_{l=0}^{m-1}(k,l)= \left\lbrace 
	\begin{array}{ll}  
	f-1&  \mbox{ f even and  k=0, }\\ 
	f-1& \mbox{ f odd and } k=\frac{m}{2} \\
	f& \mbox{ otherwise.}
	\end{array} \right.
$
\item[(iii)]  
$
	\dd \sum_{k=0}^{m-1}(k,k+n)= \left\lbrace 
	\begin{array}{ll}  
	f-1&  \mbox{ n=0, } \\
	f& \mbox{ otherwise.}
	\end{array} \right.
$
\end{itemize}
\end{lemma}

\begin{theorem}\label{thm:m_ary-const2} 
	Let $q=4f+1$ be a prime. 
	Let $\mathfrak {D}_0=D_0\cup D_2$ and $\mathfrak{D}_1=D_1\cup D_3$, and let $\underline{a}=(a_0,a_1,\ldots,a_{q-1},\ldots)$ be a quaternary sequence of period $q$ defined by 
		\be \nn
	\underline{a}_i= \left\lbrace \begin{array}{ll}  \zeta_4^k&  \mbox{if } i\in \mathfrak{D}_k ,\\  0& otherwise,
	\end{array} \right.
	\ee
	where $k=0,1$.
	Then, the sequence $\underline{a}$ is an almost quaternary NPS of type $\gamma=\frac{q-3}{2}$ with one zero-symbol. 
\end{theorem}

\begin{proof}
 The  quaternary sequence $\underline{a}$ has the autocorrelation coefficients
\be \nn 
\begin{array}{ll}
C_{\underline{a}}(t)&=\dd \sum_{i=0}^{q-1}{a_i\overline{a_{i+t}}} \\
&=\dd \sum_{k=0}^{1}|\ell : \ell \in \mathfrak{D}_{k},\ell +t \in \mathfrak{D}_{k}|+\zeta_4 |\ell : \ell \in \mathfrak{D}_{1},\ell +t \in \mathfrak{D}_{0}|+ \zeta_4^3|\ell : \ell \in \mathfrak{D}_{0},\ell +t \in \mathfrak{D}_{1}|.
\end{array}
\ee
Define $\Lambda_1$ and $\Lambda_2$ as $\Lambda_1 := |\ell : \ell \in \mathfrak{D}_{1},\ell +t \in \mathfrak{D}_{0}|$ and $\Lambda_2 := |\ell : \ell \in \mathfrak{D}_{0},\ell +t \in \mathfrak{D}_{1}|$ for some $t=1,2,\ldots,q-1$.
Then, \be \label{eqn:sum_dk} \sum_{k=0}^{1}|\ell : \ell \in \mathfrak{D}_{k},\ell +t \in \mathfrak{D}_{k}|+\Lambda_1+ \Lambda_2=4f-1.\ee 
If $t\in D_n$ for some $n=0,1,2,3$, then
\be \nn
\begin{array}{ll}
  \Lambda_1
  & =|\ell : \ell \in D_1\cup D_3,\ell +t \in D_0\cup D_2|\\
  &=|((D_1\cup D_3)+t)\cap ( D_0\cup D_2)|\\
  &=|((t^{-1}D_1+1)\cap t^{-1}D_0)\cup((t^{-1}D_1+1)\cap t^{-1}D_2)\\
  & \qquad \cup((t^{-1}D_3+1)\cap t^{-1}D_0)\cup((t^{-1}D_3+1)\cap t^{-1}D_2)| \\
  &=(1-n,0-n)+(1-n,2-n)+(3-n,0-n)+(3-n,2-n)\\
  &=(3+n,3)+(3+n,1)+(1+n,1)+(1+n,3),
\end{array}
\ee
where the last equality follows from Lemma \ref{prop:cyc_nbr}-(i). 
If $n\equiv 0 \mod{2} $, then, by using Lemma \ref{prop:cyc_nbr}-(i) and (ii), we have\be \nn
\begin{array}{ll} \dd
   \Lambda_1  =(3,1)+(3,3)+(1,1)+(1,3)=
\sum_{i=0}^{3}(1,i)=f.
    \end{array}
\ee
If $n\equiv 1 \mod{2} $, then 
\be \nn \dd
   \Lambda_1 =(0,1)+(0,3)+(2,1)+(2,3).
\ee
By Lemma \ref{prop:cyc_nbr}-(i) and (iii), 
\be \nn
\begin{array}{ll} \dd
   \sum_{k=0}^{3}(k,k+1)+\sum_{k=0}^{3}(k,k+3)  = \sum_{i=0}^{3}(1,i)+ (0,1)+(0,3)+(2,1)+(2,3)  = 2f. 
\end{array}
\ee
We know that $\sum_{i=0}^{3}(1,i) = f$ by Lemma \ref{prop:cyc_nbr}-(ii).  So, similar to case $n\equiv 0 \mod{2} $,  we get $\Lambda_1=f$ for $n\equiv 1 \mod{2} $.

Similarly, we can see $\Lambda_2=f$ for $n=0,1,2,3$. Thus, by \eqref{eqn:sum_dk}, we get $\sum_{k=0}^{1}|\ell : \ell \in \mathfrak{D}_{k},\ell +t \in \mathfrak{D}_{k}| =2f-1$. Hence, \be \nn \dd
  C_{\underline{a}}(t)=2f-1+\zeta_4 f+ \zeta_4^3 f = 2f-1.
\ee
This completes the proof.
\end{proof}

\begin{example}
	Let $q=29$, $f=7$, and $\alpha= 2$. Then,
	\be \nn
	\begin{array}{ll} 
		D_0 = \{1, 16, 24, 7, 25, 23, 20\}, \\
		D_1 = \{2, 3, 19, 14, 21, 17, 11\}, \\
		D_2=\{4, 6, 9, 28, 13, 5, 22\}, \\
		D_3=\{8, 12, 18, 27, 26, 10, 15\}.\\
	\end{array}
	\ee By Theorem \ref{thm:m_ary-const2}, we get 
	$\underline{a}= (0, 1, \zeta_4, \zeta_4, 1, 1, 1, 1, \zeta_4, 1, \zeta_4, \zeta_4, \zeta_4, 1, \zeta_4, \zeta_4, 1, \zeta_4, \zeta_4, \zeta_4, 1, \zeta_4, 1, 1, 1,1, \zeta_4, \linebreak  \zeta_4, 1,\ldots).$	
	The sequence $\underline{a}$  is an almost quaternary NPS of type $13$ with one zero-symbol.
	
	On the other hand, if we take $\zeta_2$ instead of $\zeta_4$ in $\underline{a}$, then we get $\underline{b}= (0, 1, \zeta_2, \zeta_2, 1, 1, 1, 1, \zeta_2, 1, \zeta_2, \linebreak \zeta_2, \zeta_2, 1, \zeta_2, \zeta_2, 1, \zeta_2, \zeta_2, \zeta_2, 1, \zeta_2, 1, 1, 1, 1, \zeta_2, \zeta_2, 1,\ldots).$
	The sequence $\underline{b}$  is an almost binary NPS of type $-1$ with one zero-symbol.
\end{example}
In the next example, we see that the method in Theorem \ref{thm:m_ary-const2} cannot be used for 8-ary or 6-ary sequences.
\begin{example}
    Let $q=17$, $f=2$, and $\alpha=6$. If we consider $\mathfrak{D}_0=D_0\cup D_2\cup D_4\cup D_6$ and $\mathfrak{D}_1=D_1\cup D_3\cup D_5\cup D_7$ then we get that $\underline{a}_1=(0, 1, 1, \zeta_8, 1, \zeta_8, \zeta_8, \zeta_8, 1, 1,\zeta_8, \zeta_8, \zeta_8, 1, \zeta_8, 1, 1)$ is an almost $8$-ary NPS of type $-4\zeta_8^3 + 4\zeta_8 + 7$ with one zero-symbol. If we consider $\mathfrak{D}_0=D_0\cup D_4$, $\mathfrak{D}_1=D_1\cup D_5$, $\mathfrak{D}_2=D_2\cup D_6$, and $\mathfrak{D}_3=D_3\cup D_7$ then we get that $\underline{a}_2=(0, 1, \zeta_8^2,\zeta_8^3, 1, \zeta_8^3, \zeta_8, \zeta_8, \zeta_8^2, \zeta_8^2, \zeta_8, \zeta_8, \zeta_8^3, 1, \zeta_8^3, \zeta_8^2, 1)$  is an  almost $8$-ary sequence with three distinct autocorrelation coefficients: $(3,-2\zeta_8^3 + 2\zeta_8 + 3,-4\zeta_8^3 + 4\zeta_8 + 3)$.
     Next, let $q=19$, $f=3$, and $\alpha=10$. If we consider $\mathfrak{D}_0=D_0\cup D_2\cup D_4$ and $\mathfrak{D}_1=D_1\cup D_3\cup D_5$ then we get that $\underline{a}_1=(0, 1, \zeta_6, \zeta_6, 1, 1, 1, 1, \zeta_6, 1, \zeta_6, 1, \zeta_6, \zeta_6, \zeta_6, \zeta_6, 1, 1, \zeta_6)$ is an almost $6$-ary NPS of type $(-\zeta_6 + 13,
\zeta_6 + 12)$ with one zero-symbol. If we consider $\mathfrak{D}_0=D_0\cup D_3$, $\mathfrak{D}_1=D_1\cup D_4$,and $\mathfrak{D}_2=D_2\cup D_5$, then we get that $\underline{a}_2=(0, 1, \zeta_6^2, \zeta_6^2, \zeta_6, \zeta_6^2, \zeta_6, 1, 1, \zeta_6, \zeta_6, 1, 1, \zeta_6, \zeta_6^2, \zeta_6, \zeta_6^2, \zeta_6^2, 1)$  is an almost $6$-ary sequence with three distinct autocorrelation coefficients: $(9,7,5)$.
\end{example}
 
\section{Conclusion}
As nearly perfect sequences (NPS) with consecutive two zero-symbols are very rare, we considered the sequences with non-consecutive zero-symbols and  only one zero-symbol in this paper. And so, we devised some methods to construct such sequences.

We firstly showed that an almost $m$-ary NPS with two non-consecutive zero-symbols for a prime number $m$ has an equivalent  difference set definition, namely $\ell$-partial direct product difference sets (PDPDS). Then we constructed a family of $\ell$-PDPDS. By using the equivalent difference set definition, we extended the results on the almost $m$-ary NPS with two consecutive zero-symbols to the non-consecutive case.   Also, we proved that an almost $m$-ary NPS with two consecutive zero-symbols is symmetric. 

Next, we presented that one can construct an almost $m$-ary NPS with  non-consecutive zero-symbols by using their multipliers and orbit combination in contrast to consecutive case. We proved that the auto-correlation coefficients of an $m$-ary sequence are invariant under some operations, so that we could construct new $m$-ary NPS from an existing one. Finally, we gave two sequence construction methods. In one of them, we constructed a family of NPS consisting of zero-symbols as many as non-zero-symbols and only two out-of phase autocorelation coefficients $0$ and $-1$. The second one has only one zero-symbol and one big out-of phase autocorelation coefficient. 

\section*{Acknowledgement}
The authors are supported by the Scientific and Technological Research	Council of Turkey (TÜBİTAK) under Project No: \mbox{116R026}. We would like to thank Ünal Güneş Tiryaki for fruitful discussions.

\bibliographystyle{splncs03}
\bibliography{ozdenyayla}  
	
\end{document}